\documentclass[oneside,12pt]{amsart}
\usepackage{amsmath,amssymb,mathrsfs,stmaryrd}

\usepackage{color}
\usepackage[
color,
final]        % per togliere showkeys sostituire draft con final
{showkeys}

\definecolor{refkeybis}{gray}{.65}%     per avere le labels stampate chiare
\definecolor{labelkeybis}{gray}{.65}%   basta modificare .50:= intensita grigio
{\makeatletter
\def\SK@refcolor{\color{refkeybis}}%
\def\SK@labelcolor{\color{labelkeybis}}}

\calclayout
\usepackage{latexsym}
\usepackage{amsmath}
\usepackage{amssymb,mathrsfs}
\newcommand{\N}{\mathbb{N}}

\newcommand{\R}{\mathbb{R}}
\newcommand{\Z}{\mathbb{Z}}
\newcommand{\BB}{\mathscr{B}}

\newcommand{\LL}{\mathscr{L}}
\newcommand{\MM}{\mathscr{M}}

\newcommand{\cM}{{\ensuremath{\mathcal M}}}

\renewcommand{\ss}{{\mbox{\boldmath$s$}}}

\newcommand{\xx}{{\mbox{\boldmath$x$}}}

\newcommand{\sss}{{\mbox{\scriptsize\boldmath$s$}}}

\newcommand{\sxx}{{\mbox{\scriptsize\boldmath$x$}}}

\newcommand{\sfd}{{\sf d}}
\newcommand{\sfe}{{\sf e}}

\newcommand{\sfi}{{\sf i}}

\newcommand{\sfp}{{\sf p}}

\newcommand{\sfV}{{\sf V}}

\newcommand{\BorelSets}[1]{\BB(#1)}

\newcommand{\vett}{\boldsymbol}
\newcommand{\XX}{{\vett X}}
\newcommand{\eps}{\varepsilon}

\newcommand{\boreliani}[1]{\mathscr B(#1)}

\newcommand{\tight}{\mathscr T}
\newcommand{\A}{\mathscr A}

\newcommand{\Prob}[2][]{{\mathscr P}_{#1}(#2)}
\newcommand{\optimalplan}[2]{\Gamma^\psi_o(#1,#2)}
\newcommand{\admissibleplan}[2]{\Gamma(#1,#2)}

\newcommand{\Leb}{\mathscr L}

\renewcommand{\div}{\mathop{\rm div}\nolimits}

\newcommand{\supp}{\mathop{\rm supp}}

\newcommand{\id}{\sfi}

\newcommand{\pV}{\sfp\sfV}
\newcommand{\eV}{\sfe\sfV}

\theoremstyle{plain}
\newtheorem{theorem}{Theorem}[section]  % per cambiare la numerazione
\newtheorem{lemma}[theorem]{Lemma}

\newtheorem{proposition}[theorem]{Proposition}

\theoremstyle{remark}
\newtheorem{remark}[theorem]{Remark}

\begin{document}
\title{Absolutely continuous curves in extended Wasserstein-Orlicz spaces}
	%\thanks{This work has been partially supported by the IMATI (CNR Pavia, Italy)}
\author{Stefano Lisini}
\address{Dipartimento di Matematica ``F.Casorati'', Universit\`a degli Studi di Pavia, 27100 Pavia, Italy}
\email{stefano.lisini@unipv.it}

\begin{abstract}
In this paper we extend a previous result of the author
\cite{lisini07} of characterization of absolutely continuous curves in Wasserstein spaces to a more general class of spaces:
the spaces of probability measures endowed with the Wasserstein-Orlicz distance constructed
on extended Polish spaces (in general non separable), recently considered in \cite{ags11}.
An application to the geodesics of this Wasserstein-Orlicz space is also given.
% \subclass{} % Mathematics Subject Classification 2000 (optional)
\end{abstract}
\maketitle
\section{Introduction}
In this paper we extend a previous result of the author \cite{lisini07} to a more general class of spaces.
The result in \cite{lisini07} concerns the representation of absolutely continuous curves with finite energy
in the Wasserstein space $(\Prob{X,\sfd},W_p)$
(the space of Borel probability measures on a Polish metric space $(X,\sfd)$, endowed with the $p$-Wasserstein distance induced by $\sfd$)
by means of superposition of curves of the same kind on the space $(X,\sfd)$.
The superposition is described by a probability measure on the space of continuous curves in $(X,\sfd)$
representing the curve in $(\Prob{X,\sfd},W_p)$ and satisfying a suitable property.

Here we extend the previous representation result in two directions:
in the first one we consider a so-called extended Polish space $(X,\tau,\sfd)$ instead of a Polish space $(X,\sfd)$;
in the second one we consider the $\psi$-Orlicz-Wasserstein distance induced by an increasing convex function $\psi:[0,+\infty)\to [0,+\infty]$ instead
of the $p$-Wasserstein distance modelled on the particular case of $\psi(r)=r^p$ for $p>1$.

The class of extended Polish spaces was introduced in the recent paper \cite{ags11}.
%where there is distinction between the topology of the space and the
%topology induced by the extended distance $\sfd$ used for the construction of the Wasserstein distance.
The authors consider a Polish space $(X,\tau)$,
i.e. $\tau$ is a separable topology on $X$ induced by a distance $\delta$ on $X$ such that $(X,\delta)$ is complete.
The Wasserstein distance is defined between Borel probability measures on $(X,\tau)$ and
constructed by means of an extended distance $\sfd$ on $X$ that can assume the value $+\infty$.
The minimization problem defining the extended Wasserstein distance makes sense between Borel probability measures on $(X,\tau)$,
assuming that the extended distance $\sfd$ is lower semi continuous with respect to $\tau$.

A typical example of extended Polish space is the abstract Wiener space $(X,\tau,\gamma)$ where  $(X,\tau)$ is a separale Banach space
and $\tau$ is the topology induced by the norm, $\gamma$ is a gaussian reference measure on $X$ with zero mean and supported on all the space.
The extended distance is given by $\sfd(x,y)=|x-y|_H$ if $x-y\in H$, where $H$ is the Cameron-Martin
space associated to $\gamma$ in $X$ and $|\cdot|_H$ is the Hilbertian norm of $H$, and $\sfd(x,y)=+\infty$ if $x-y\not\in H$
(see for instance \cite{Stroock11}).

The Wasserstein-Orlicz distance is still unexplored.
At the author's knowledge, only the papers \cite{Sturm4} and, more recently, \cite{Kuwada13} deal with this kind of spaces.
In the paper \cite[Remark 3.19]{FGY11}, the authors discuss the possibility to use this kind of Wasserstein-Orlicz distance
to extend their results for equation of the form $\partial_t u - \div(u\nabla H(u^{-1}\nabla u)=0$
to the case of a convex function $H$ with non power growth.

Only the particular case of the Wasserstein-Orlicz distance $W_\infty$,
corresponding to the function $\psi(s)=0$ if $s\in[0,1]$ and $\psi(s)=+\infty$ if $s\in(1,+\infty)$
%(the corresponding Orlicz norm is the $L^\infty$ norm)
has been deeply investigated.
The extension of the representation Theorem of \cite{lisini07} to the $W_\infty$ case has been proved in \cite{AmbrosioDiMarino12}.
Another refinement of the representation Theorem of \cite{lisini07} is contained in \cite[Sec. 5]{BS11}.
The problem of the validity of the representation Theorem of \cite{lisini07} in the case of a general Wasserstein-Orlicz space
is raised in the last section of \cite{ags13}.

%Two extension was already given in \cite{lisini07}:
%\begin{itemize}
%\item[1)]The case of an extended distance $\sfd$ (called pseudo distance in \cite{lisini07})
%where $\sfd$ satisfies all the axioms of definition of a distance but $\sfd$ can assume the value $+\infty$.
%In this situation, only the separable and complete $(X,\sfd)$ case was considered (see Corollary 1 in \cite{lisini07}).
%\item[2)]The case of $(X,\sfd)$ non separable with the condition that the curves takes the values
%in a proper subspace of the Borel probability measures on $X$: precisely the set of tight measures (see Corollary 3 in \cite{lisini07}).
%We recall that Prokhorov's Theorem does not hold in non separable spaces and there exist Borel probability measures that are not tight.
%\end{itemize}
%In both cases the topology considered on $X$ was the one induced by the distance (or extended distance) $\sfd$ used for the
%construction of the Wasserstein distance.

For the precise statement of the result we address to Theorem \ref{th:curveAC}.
The strategy of the proof is similar to the one used to prove Theorem 5 of \cite{lisini07},
but there are several additional difficulties because $(X,\sfd)$, in general, is non separable and
the function $\psi$ that induces the Wasserstein-Orlicz distance is not homogeneous.

The paper is structured as follows:
in Section \ref{notation} we introduce the framework of our study and some preliminary results,
in Section \ref{main} we state and prove the main theorem of the paper, and
finally in Section \ref{sec:geodesics} we apply the main theorem in order to characterize the geodesics of the
Wasserstein-Orlicz space.

%In subsection \ref{subsec:PolishSpaces} we give the definition of extended Polish spaces and
%the space of Borel probability measures on this class of spaces.
%In subsection \ref{subsec:ContCurves} we give the definition of absolutely continuous curves on extended Polish spaces,
%the principal properties and a compactness result.
%In subsection \ref{subsec:MisurFuncSpace} we recall a compactness result on the space of measurable functions.
%In subsection \ref{subsec:PushForw} we give the definition of image measure and we prove an important
%Lemma for the proof of the main Theorem of the paper.
%In subsection \ref{subsec:Wasserstein} we define the Wasserstein-Orlicz extended distance.

\section{Notation and preliminary results}\label{notation}
\subsection{Extended Polish spaces and probability measures}\label{subsec:PolishSpaces}
Given a set $X$, we say that $\sfd:X\times X \to [0,+\infty]$ is an extended distance if
\begin{itemize}
\item $\sfd(x,y)=\sfd(y,x)$ for every $x,y \in X$,
\item $\sfd(x,y)=0$ if and only if $x=y$,
\item $\sfd(x,y)\leq \sfd(x,z)+\sfd(z,y)$ for every $x,y,z \in X$.
\end{itemize}
$(X,\sfd))$ is called extended metric space.
We observe that the only difference between a distance and an extended distance is that $\sfd(x,y)$ could be equal to $+\infty$.

We say that $(X,\tau,\sfd)$ is a Polish extended space if:
\begin{itemize}
\item[(i)] $\tau$ is a topology on $X$ and $(X,\tau)$ is Polish, i.e. $\tau$ is induced by a distance $\delta$
    such that the metric space $(X,\delta)$ is separable and complete;
\item[(ii)] $\sfd$ is an extended distance on $X$ and $(X,\sfd)$ is a complete extended metric space;
\item[(iii)] For every sequence $\{x_n\}\subset X$ such that $\sfd(x_n,x)\to 0$ with $x\in X$, we have that $x_n\to x$ with respect to the
topology $\tau$;
\item[(iv)] $\sfd$ is lower semicontinuous in $X\times X$, with
respect to the $\tau\times\tau$ topology; i.e.,
\begin{equation}\label{lscd}
  \liminf_{n\to+\infty} \sfd(x_n,y_n)\geq\sfd(x,y), \qquad \forall (x,y)\in X\times X, \quad \forall (x_n,y_n)\to (x,y) \text{ w.r.t. } \tau\times\tau.
\end{equation}
\end{itemize}

In the sequel, the class of compact sets,
the class of Borel sets $\BorelSets{X}$, the class
$C_b(X)$ of bounded continuous functions and the class $\Prob{X}$ of
Borel probability measures, are always referred to the topology
$\tau$, even when $\sfd$ is a distance.

We say that a sequence $\mu_n \in \Prob{X}$ \emph{narrowly converges} to $\mu \in \Prob{X}$ if
\begin{equation}
	\lim_{n \rightarrow +\infty} \int_X \varphi (x)\,d\mu_n(x)=\int_X \varphi (x)\,d\mu(x)
			\qquad \forall \varphi \in C_b(X).
\end{equation}

It is well known that the narrow convergence is induced by a distance on $\Prob{X}$ (see for instance \cite[Remark 5.1.1]{ags})
and we call \emph{narrow topology} the topology induced by this distance.
In particular the compact subsets of $\Prob{X}$ coincides with
sequentially compact subsets of $\Prob{X}.$\\
We also recall that if $\mu_n \in \Prob{X}$ narrowly converges to $\mu \in \Prob{X}$ and $\varphi:X \to (-\infty,+\infty]$
is a lower semi continuous (with respect to $\tau$) function bounded from below, then
\begin{equation}\label{narrowlsc}
	\liminf_{n \rightarrow +\infty} \int_X \varphi (x)\, d\mu_n(x) \geq \int_X \varphi (x)\, d\mu(x) .
\end{equation}

A subset $\tight \subset \Prob{X}$ is said to be tight if
\begin{equation}\label{tightness}
	\forall \varepsilon >0 \quad \exists K_\varepsilon \subset X \mbox{ compact}:
	 \mu (X\setminus K_\varepsilon)<\varepsilon \quad \forall \mu \in \tight ,
\end{equation}
or, equivalently, if there exists a function $\varphi:X \rightarrow [0,+\infty]$ with compact sublevels
$\lambda_c(\varphi):=\{x \in X: \varphi(x)\leq c\}$,  such that
\begin{equation} \label{tightboundedness}
	\sup_{\mu \in \tight} \int_X\varphi(x) \, d\mu(x)<+\infty .
\end{equation}

By Prokhorov Theorem, a
set $\tight \subset \Prob{X}$ is tight if and only if $\tight$ is relatively compact in $\Prob{X}.$
In particular, the Polish condition on $\tau$ guarantees that all Borel probability
measures $\mu\in \Prob{X}$ are tight.

\subsection{Orlicz spaces}\label{subsec:OrliczSpaces}
Given
\begin{equation}\label{hp:psi}
\begin{aligned}
\psi:[0,+\infty)\to [0,+\infty] & \text{ convex, lower semicontinuous, non-decreasing, } \psi(0)=0,\\
&\quad \lim_{x\to +\infty}\psi(x) = +\infty,
\end{aligned}
\end{equation}
%\begin{equation}\label{superlinearity}
%$\lim_{x\to 0}\frac{\psi(x)}{x} = 0, \qquad
%\lim_{x\to +\infty}\frac{\psi(x)}{x} = +\infty.
%\end{equation}
a measure space $(\Omega,\nu)$ and a $\nu$-measurable function $u:\Omega\to\R$,
the $L^\psi_\nu(\Omega)$ Orlicz norm of $u$ is defined by
$$\| u\|_{L^\psi_\nu(\Omega)}:=\inf \left\{\lambda>0: \int_\Omega \psi\Big(\frac{|u|}{\lambda}\Big)\,d\nu \leq 1\right\}.$$
The Orlicz space $L^\psi_\nu(\Omega):=\{u:\Omega \to \R, \text { measurable} : \|u\|_{L^\psi_\nu(\Omega)}<+\infty \}$
is a Banach space.
For the theory of the Orlicz spaces we refer to the complete monography \cite{RaoRen}.

%%%%%%%%%%%%%%%%........................ESEMPI....................................................

%If $\Omega$ is an open subset of $\R^n$ and $\nu$ is a Borel measure, we denote by $E^\psi_\nu(\Omega)$ the closure of
%$C^\infty_0(\Omega)$ in  $L^\psi_\nu(\Omega)$.

%If $L\in (E^\psi_\nu(\Omega))^*$, then there exists a unique $v \in L^{\psi^*}_\nu(\Omega)$
%such that $Lu=\int_\Omega v u \,d\nu$ for every $u\in E^\psi_\nu(\Omega)$
%and $\| L \| \leq 2 \|v\|_{L^{\psi^*}_\nu(\Omega)}$.

%We observe that the monotonicity property holds: $\| u\|_{L^\psi_\nu(\Omega)}\leq \| v\|_{L^\psi_\nu(\Omega)}$ if
%$|u| \leq |v|$.
%
Given a bounded sequence $\{w_n\}\subset L^\psi_\nu(\Omega)$, the following property of lower semi continuity of the norm holds:
\begin{equation}\label{lscOrliczNorm}
 \liminf_{n\to\infty} w_n(x) \geq w(x) \quad\text{for $\nu$-a.e. } x\in \Omega \quad
 \Longrightarrow \quad \liminf_{n\to\infty} \|w_n\|_{L_\nu^\psi(\Omega)} \geq \|w\|_{L_\nu^\psi(\Omega)}.
\end{equation}
Indeed, denoting by $\lambda_n:=\|w_n\|_{L_\nu^\psi(\Omega)}$ and $\lambda:=\liminf_n\lambda_n$, up to extract a subsequence we can assume that
$\lambda=\lim_n\lambda_n$. By the lower semicontinuity and the monotonicity of $\psi$ we have
\begin{equation*}
    \liminf_{n\to\infty}\psi\Big(\frac{w_n(x)}{\lambda_n}\Big) \geq \psi\Big(\frac{w(x)}{\lambda}\Big) \qquad\text{for $\nu$-a.e. }x\in \Omega.
\end{equation*}
Finally, by Fatou's lemma
$$ 1\geq \liminf_{n\to\infty}\int_\Omega\psi\Big(\frac{w_n(x)}{\lambda_n}\Big)\,d\nu(x) \geq \int_\Omega\psi\Big(\frac{w(x)}{\lambda}\Big)\,d\nu(x) $$
which shows that $\lambda \geq \|w\|_{L_\nu^\psi(\Omega)}$.

We denote by $\psi^*:=[0,+\infty)\to [0,+\infty]$ the conjugate of $\psi$ defined by $\psi^*(y)=\sup_{x\geq 0}\{xy-\psi(x)\}$.
The following generalized H\"older's inequality holds
\begin{equation}\label{holder}
 \int_\Omega u(x)v(x)\,d\nu(x) \leq 2\|u\|_{L_\nu^\psi(\Omega)}\|v\|_{L_\nu^{\psi^*}(\Omega)},
\end{equation}
and the following equivalence between the Orlicz norm in $L_\nu^\psi(\Omega)$ and the dual norm of $L_\nu^{\psi^*}(\Omega)$  holds
\begin{equation}\label{duality}
  \|u\|_{L_\nu^\psi(\Omega)} \leq \sup\Big\{\int_\Omega |u(x)v(x)| \,d\nu(x): v\in L_\nu^{\psi^*}(\Omega), \|v\|_{L_\nu^{\psi^*}(\Omega)}\leq 1  \Big\} \leq 2\|u\|_{L_\nu^\psi(\Omega)}.
\end{equation}

In the statement of our main theorem we will assume, in addition to \eqref{hp:psi},
that $\psi$ is superlinear at $+\infty$, i.e.
\begin{equation}\label{superlinearity}
\lim_{x\to +\infty}\frac{\psi(x)}{x} = +\infty,
\end{equation}
and it has null right derivative at $0$, i.e.
\begin{equation}\label{beatzero}
\lim_{x\to 0}\frac{\psi(x)}{x} = 0.
\end{equation}
It is easy to check that conditions \eqref{superlinearity} and \eqref{beatzero} are equivalent to assume that
$\psi^*(y)>0$ and $\psi^*(y)<+\infty$ for every $y>0$.

Typical examples of admissible $\psi$ satisfying \eqref{hp:psi}, \eqref{superlinearity} and \eqref{beatzero} are:
\begin{itemize}
\item $\psi(x)=x^p$ for $p\in(1,+\infty)$ and the corresponding Orlicz norm is the standard $L^p$ norm;
\item $\psi(x)=0$ if $x\in[0,1]$ and $\psi(x)=+\infty$ if $x\in(1,+\infty)$ and the corresponding Orlicz norm is the $L^\infty$ norm;
\item $\psi(x)=e^x-x-1$, exponential growth;
\item $\psi(x)=e^{x^p}-1$ for $p\in(1,+\infty)$, power exponential growth;
\item $\psi(x)=(1+x)\ln(1+x)-x$, $L\log L$-growth.
\end{itemize}

\subsection{Continuous curves}\label{subsec:ContCurves}
Given $(X,\tau,\sfd)$ an extended Polish space, $I:=[0,T]$, $T>0$, we denote by
$C(I;X)$ the space of continuous curves in $X$ with respect to the topology $\tau$.
$C(I;X)$ is a Polish space with the metric
\begin{equation}\label{distanceGamma}
	\delta_{\infty}(u, \tilde{u})=\sup_{t\in I}\delta(u(t), \tilde u(t)),
\end{equation}
where $\delta$ is a complete and separable metric on $X$ inducing $\tau$.

Given $\psi$ satisfying \eqref{hp:psi},
we say that a curve $u:I \rightarrow X$ belongs to $AC^\psi(I;(X,\sfd)),$ if there exists $m \in L^\psi(I)$ such that
\begin{equation}\label{AC}
	\sfd(u (s),u (t)) \leq \int _s^t m(r)\,dr \qquad \forall s,t \in I, \quad s \leq t.
\end{equation}
We also denote by $AC(I;(X,\sfd))$ the set $AC^\psi(I;(X,\sfd))$ for $\psi(r)=r$.
We call a curve $u \in AC^\psi(I;(X,\sfd))$ an
absolutely continuous curve with finite $L^\psi$-energy.

It can be proved that (see \cite[Theorem 1.1.2]{ags}) for every $u \in AC^\psi(I;(X,\sfd))$, there exists the following limit, called metric derivative,
	\begin{equation}
		|u'|(t):= \lim_{h \rightarrow 0} \frac{\sfd(u (t+h),u (t))}{|h|} \qquad \text{ for $\Leb ^1$-a.e. $t \in I$}.
	\end{equation}
The function $t \mapsto |u'|(t)$ belongs to $L^\psi(I)$ and it is the minimal one that satisfies \eqref{AC}.

%If $u \in AC^\psi(I;(X,\sfd))$
%and $\psi$ is superlinear
%\begin{equation}\label{superlinearity}
%\lim_{x\to +\infty}\frac{\psi(x)}{x} = +\infty,
%\end{equation}
%then
%\begin{equation}\label{modcont}
%    \sfd(u(s),u(t)) \leq \| |u'| \|_{L^\psi(I)} \omega (|s-t|), \qquad \forall\; s,t \in I,
%\end{equation}
%where the modulus of continuity $\omega:[0,+\infty)\to \R$ is defined by
%\begin{equation}\label{def:omega}
% \omega(t):= t \psi^{-1}(1/t) \text{ for $t>0$ and } \omega(0)=0
%\end{equation}
%(we observe that $\psi^{-1}(s)$ is well defined for every $s>0$ and in the case of $\psi(r)=+\infty$ for $r > r_0$ and $\psi(r_0) < +\infty$
%we define $\psi^{-1}(s)= r_0$ for every $s>\psi(r_0)$,
%then \eqref{superlinearity} implies that $\lim_{t\to 0} \omega(t)=0$).
%Indeed for $s<t$ the Jensen's inequality reads
%$$ \psi\Big(\frac{1}{t-s}\int_s^t \frac{|u'|(r)}{\||u'|\|_{L^\psi(s,t)}}\,dr\Big) \leq \frac{1}{t-s}\int_s^t \psi\Big(\frac{|u'|(r)}{\||u'|\|_{L^\psi(s,t)}}\Big)\,dr \leq \frac{1}{t-s}$$
%and consequently
%$$\int_s^t |u'|(r)\,dr \leq (t-s)\psi^{-1}\Big(\frac{1}{t-s}\Big)\| |u'| \|_{L^\psi(s,t)}$$
%and \eqref{modcont} follows taking into account the monotonicity of the $L^\psi$ norm.

The following Lemma will be useful in the proof of our main theorem.
\begin{lemma}\label{lemma:dq} Let $\psi$ be satisfying \eqref{hp:psi}, \eqref{superlinearity} and \eqref{beatzero}.
If $u:I\to (X,\sfd)$ is right continuous at every point and continuous except at most a countable set, and
\begin{equation}\label{eq:sobolev}
   \limsup_{h\to 0^+} \left\|\frac{\sfd(u(\cdot+h),u(\cdot))}{h}\right\|_{L^\psi(I)} <+\infty,
\end{equation}
where $u$ is extended for $t>T$ as $u(t)=u(T)$,
then $u \in AC^\psi(I;(X,\sfd))$.
\end{lemma}
\begin{proof}
Since $I$ is bounded, by the assumptions on $u$ we have that the $\sfd$-closure of $u(I)$ is compact in $(X,\sfd)$.
Consequently $u(I)$ is $\sfd$-separable.
We consider a sequence $\{y_n\}_{n \in \N}$ dense in $(u(I),\sfd)$.
We fix $n\in\N$. Defining $u_n:I\to\R$ by $u_n(t):=\sfd(u(t),y_n)$, the triangular inequality implies
\begin{equation}\label{in:ab}
	|u_n(t+h)-u_n(t)| \leq \sfd(u(t+h),u(t)), \qquad \forall t\in I, h>0.
\end{equation}
Given a test function $\eta \in C^\infty_c(I)$ and $h>0$, recalling H\"older inequality \eqref{holder} we obtain
\begin{align*}
\left|\int_Iu_n(t)\frac{\eta(t-h)-\eta(t)}{h}\,dt\right| & = \left|\int_I\eta(t)\frac{u_n(t+h)-u_n(t)}{h}\,dt\right| \\
    &\leq 2 \left\|\frac{u_n(\cdot+h)-u_n(\cdot))}{h}\right\|_{L^\psi(I)} \left\|\eta\right\|_{L^{\psi^*}(I)}.
\end{align*}
By the last inequality, \eqref{eq:sobolev} and \eqref{in:ab}, passing to the limit for $h\to 0$ we have that
\begin{equation}\label{Lbounded}
	\left|\int_I u_n(t)\eta'(t)\,dt\right|
    \leq C \left\|\eta\right\|_{L^{\psi^*}(I)}.
\end{equation}
The linear functional $\LL_n:(C^\infty_c(I),\|\cdot\|_{L^{\psi^*}(I)})\to\R$
defined by $\LL_n(\eta)=\int_I u_n(t)\eta'(t)\,dt$, by \eqref{Lbounded},
is bounded and we still denote by $\LL_n$ its
extension to $E^{\psi^*}(I)$, the closure of $C^\infty_c(I)$ with respect to the norm $\|\cdot\|_{L^{\psi^*}(I)}$.
Since, by \eqref{superlinearity} and \eqref{beatzero}, $\psi^*$ is continuous and strictly positive on $(0,+\infty)$, $\LL_n$ is uniquely represented
by an element $v_n\in L^{\psi^{**}}(I)$ (see Theorem 6, pag. 105 of \cite{RaoRen}).
The element $v_n$ coincides with the distributional derivative of $u_n$ and then
$u_n \in AC^{\psi}(I;\R)$ (we observe that $\psi^{**}=\psi$ because $\psi$ is convex and lower semi continuous).
%(see for instance \cite{brezis} Proposition VIII.3).
%Since $u_n$ is also continuous, then $u_n$ is absolutely continuous and, consequently, a.e. differentiable.\\
We denote by $u'_n(t)$ the pointwise derivative of $u_n$ which exists for a.e. $t\in I$.

We introduce the negligible set
 $$N=\bigcup_{n \in \N} \{t\in I : u'_n(t) \mbox{ does not exists} \},$$
and we define $m(t):=\sup_{n\in\N} |u_n'(t)|$ for all $t \in I \setminus N.$
By the density of $\{y_n\}_{n\in\N}$ in $u(I)$, we have that for all $t,s \in I$, with $s<t$,
\begin{equation}\label{in:abc}
	\sfd(u(t),u(s)) = \sup_{n\in\N} |u_n(t)- u_n(s)| \leq \sup_{n\in\N} \int_s^t |u_n'(r)|\, dr
		\leq \int_s^t m(r)\, dr.
\end{equation}
We show that $m \in L^\psi(I).$ Actually, by (\ref{in:ab}), if $t \in I \setminus N$ then
\begin{equation*}
	|u_n'(t)|=\lim_{h \to 0^+}\frac{|u_n(t+h)- u_n(t)|}{h}
		\leq \liminf_{h\rightarrow 0^+} \frac{\sfd(u(t+h),u(t))}{h},
\end{equation*}
which implies $\displaystyle{m(t) \leq \liminf_{h\rightarrow 0^+} \frac{\sfd(u(t+h),u(t))}{h}}$.
By \eqref{eq:sobolev} and \eqref{lscOrliczNorm} we conclude.
\end{proof}

%The following useful compactness result is proved in \cite{ags}, Proposition 3.3.1.
%\begin{theorem}\label{th:compAA}
%Let $K\subset X$ compact and $u_n:I\to X$ curves satisfying
%\begin{equation}
% u_n(t)\in K \qquad \forall n\in\N, \quad \forall t \in I,
%\end{equation}
%\begin{equation}
% \limsup_{n\to\infty}\sfd(u_n(t),u_n(s))\leq \tilde\omega(s,t)  \qquad \forall s,t \in I,
%\end{equation}
%for a symmetric function $\tilde\omega:I\times I \to [0,+\infty)$, such that
%\begin{equation}
% \limsup_{(s,t)\to(r,r)}\tilde\omega(s,t)=0  \qquad \forall r \in I.
%\end{equation}
%Then there exist a subsequence $u_{n_k}$ and a limit curve $u:I\to X$ such that
%\begin{equation}
%  u_{n_k}(t) \to u(t) \quad \text{w.r.t. }\tau \qquad \forall t \in I,
%\end{equation}
%\begin{equation}
% u\in C(I;(X,\sfd)).
%\end{equation}
%\end{theorem}

\subsection{The $\cM(I;X)$ space}\label{subsec:MisurFuncSpace}
We denote by $\MM(I;X)$ the space of curves $u:I \to X$ which are Lebesgue measurable as functions with values in $(X,\tau)$.
We denote by $\cM(I;X)$ the quotient space of $\MM(I;X)$ with respect to the equality $\Leb^1$-a.e. in $I$.
The space $\cM(I;X)$ is a Polish space endowed with the metric
$$\delta_1(u,v):=\int_0^T \tilde\delta(u(t),v(t))\,dt,$$
where $\tilde\delta(x,y):=\min\{\delta(x,y),1\}$ is a bounded distance still inducing $\tau$ and
$\delta$ is a distance inducing $\tau$.

The space $\cM(I;X)$ coincides with $L^1(I;(X,\tilde\delta))$.
It is well known that $\delta_1(u_n,u)\to 0$ as $n\to+\infty$ if and only if $u_n \to u$ in measure as $n\to+\infty$; i.e.
$$ \lim_{n\to+\infty}\Leb^1(\{t\in I: \delta(u_n(t),u(t))>\sigma\})=0, \qquad \forall \, \sigma>0.$$

We recall a useful compactness criterion in $\cM(I;X)$, \cite[Theorem 2]{RossiSavare}.
\begin{theorem}\label{th:complp}
A family $\A \subset \cM(I;X)$ is precompact if
there exists a function $\Psi:X\to [0,+\infty]$ whose
sublevels $\lambda_c(\Psi):= \{ x\in X : \Psi (x) \leq c\}$ are compact for every $c \geq 0,$
such that
\begin{equation}\label{t}
	\sup_{u\in \A}\int_0^T \Psi (u(t)) \, dt <+\infty,
\end{equation}
and there exists a map $g:X\times X \to [0,\infty]$
lower semi continuous with respect to $\tau\times\tau$ such that
$$ g(x,y)=0 \quad\Longrightarrow \quad x=y$$
and
\begin{equation*}
	\lim_{h\to 0^+} \sup_{u\in \A} \int_0 ^{T-h} g(u(t+h),u(t))\, dt =0.
\end{equation*}
\end{theorem}

\subsection{Push forward of probability measures}\label{subsec:PushForw}
If $Y,Z$ are topological spaces, $\mu \in \Prob{Y}$ and $F:Y \rightarrow Z$ is a Borel map (or a $\mu$-measurable map),
the \emph{push forward of $\mu$ through F}, denoted by $F_\# \mu \in \Prob{Z}$, is defined as follows:
\begin{equation}\label{pushforward}
	F_\# \mu (B):=\mu (F^{-1}(B)) \qquad \forall B \in \boreliani{Z}.
\end{equation}
It is not difficult to check that this definition is equivalent to
\begin{equation}\label{changevariable}
	\int _Z \varphi (z)\,d(F_\# \mu )(z)=\int _Y \varphi (F(y))\,d\mu (y) %\qquad \forall \varphi \in C_b(Z).
\end{equation}
for every bounded Borel function $\varphi :Z \rightarrow \R$.
More generally (\ref{changevariable}) holds for every $F_\# \mu$-integrable
function $\varphi :Z \rightarrow \R.$ \\
We recall the following composition rule: for every $\mu \in \Prob{Y}$ and for all Borel maps
$F:Y \rightarrow Z$ and $G:Z \rightarrow W$, we have
\begin{equation*}%\label{compositionrule}
	(G \circ F)_\# \mu =G_\# (F_\# \mu ).
\end{equation*}
The following continuity property holds:
\begin{equation*}
	F:Y \rightarrow Z \quad \mbox{continuous} \quad \Longrightarrow \quad
	F_\# :\Prob{Y} \rightarrow \Prob{Z} \quad \mbox{narrowly continuous.}
\end{equation*}

We say that $\mu\in \Prob{Y}$ is concentrated on the set $A$ if $\mu(X\setminus A)=0$.
It follows from the definition that $F_\#\mu$ is concentrated on $F(A)$ if $\mu$ is concentrated on $A$.

The support of a Borel probability measure $\mu\in \Prob{Y}$ is the closed set defined by
$\supp \mu =\{y\in Y: \mu(U)>0, \forall U \text{ neighborhood of }y \}$.
 %where $B_r(x)$ denotes the open ball with respect to a distance $\delta$ inducing $\tau$.
$\mu$ is concentrated on $\supp\mu$ and it is the smallest closed set on which $\mu$ is concentrated.

In general we have $F(\supp \mu) \subset \supp F_\#\mu \subset \overline{F(\supp \mu)}$ for $F:Y \rightarrow Z$ continuous.

It follows that $F_\#\mu(\supp F_\#\mu\setminus F(\supp \mu))=0$.

The following Lemma is fundamental in our proof of Theorem \ref{th:curveAC}.
It allows to recover a pointwise bound assuming an integral bound.
\begin{lemma}\label{lemma:Kur}
Let $Y$ be a Polish space and $\{\mu_n\}_{n\in\N}\subset \Prob{Y}$ be a sequence narrowly convergent to $\mu\in\Prob{Y}$ as $n\to+\infty$.
Let $F_n:Y\to [0,+\infty)$ be a sequence of $\mu_n$-measurable functions such that
\begin{equation}\label{integralbound}
  \sup_{n\in\N}\int_Y F_n(y)\,d\mu_n(y)<+\infty.
\end{equation}
Then there exists a subsequence $\mu_{n_k}$ such that
\begin{equation}\label{pointwisebound}
 \text{for $\mu$-a.e. } \bar y\in\supp{\mu} \quad \exists y_{n_k}\in\supp{\mu_{n_k}} : \lim_{k\to+\infty}y_{n_k}=\bar y
 \qquad \text{and } \quad \sup_{k\in\N}F_{n_k}(y_{n_k})<+\infty.
\end{equation}
\end{lemma}
\begin{proof}
Let us define the sequence $\nu_n:=(\id\times F_n)_\#\mu_n\in\Prob{Y\times\R}$, where $\id$ denotes the identity map in $Y$.
We denote by $\pi^1:Y\times \R \to Y$ and $\pi^2:Y\times \R \to \R$ the projections defined by $\pi^1(y,z)=y$ and $\pi^2(y,z)=z$.
The set $\{\nu_n\}_{n\in\N}$ is tight because $\{\pi^1_\#\nu_n\}_{n\in\N}$ and $\{\pi^2_\#\nu_n\}_{n\in\N}$ are tight.
Indeed $\pi^1_\#\nu_n=\mu_n$ is narrowly convergent,
and $\pi^2_\#\nu_n=(F_n)_\#\mu_n$ has first moments uniformly bounded because
$$\int_\R |z|\,d\pi^2_\#\nu_n(z)=\int_Y |F_n(y)|\,d\mu_n(y),$$
$F_n\geq 0$ and \eqref{integralbound} holds.
By Prokhorov's Theorem there exists $\nu\in\Prob{Y\times\R}$ and a subsequence $\{\nu_{n_k}\}_{k\in\N}\subset\Prob{Y\times\R}$
narrowly convergent to $\nu$.
Since $\pi^1_\#\nu_n=\mu_n$ and $\pi^1_\#\nu_{n_k} \to \pi^1_\#\nu$ as $k\to +\infty$ we have that $\pi^1_\#\nu=\mu$.

Let $\bar y\in\pi^1(\supp\nu)$, and we observe that $\mu(\supp \mu \setminus \pi^1(\supp\nu))=0$.
By definition of $\bar y$ there exists $z\in\R$ such that $(\bar y,z)\in\supp{\nu}$.
Let $h\in\N$ and $D_{1/h}(\bar y,z):=B_{1/h}(\bar y)\times (z-1/h,z+1/h)$ where $B_r(\bar y)$ denotes the open ball of radius $r$ and center $\bar y$,
when a distance in $Y$ is fixed.
By \eqref{narrowlsc}, with $\varphi$ the characteristic function of $D_{1/h}(\bar y,z)$, we obtain
$$ \liminf_{k\to+\infty}\nu_{n_k}(D_{1/h}(\bar y,z))\geq \nu(D_{1/h}(\bar y,z))>0.$$
Then there exists $k(h)\in\N$ such that
\begin{equation}\label{abc}
    \nu_{n_k}(D_{1/h}(\bar y,z))>0 \qquad \forall k\geq k(h).
\end{equation}
By definition of $\nu_n$
\begin{equation}\label{abcd}
\begin{aligned}
    \nu_{n_k}(D_{1/h}(\bar y,z)) &= \mu_{n_k}(\{y\in Y: (\id\times F_{n_k})(y)\in D_{1/h}(\bar y,z)\})\\
    &= \mu_{n_k}(\{y\in Y: (y,F_{n_k}(y))\in B_{1/h}(\bar y)\times (z-1/h,z+1/h)\}).
\end{aligned}
\end{equation}
By \eqref{abc} and \eqref{abcd} we have that
\begin{equation}\label{ab}
    \supp \mu_{n_k} \cap \{y\in Y: (y,F_{n_k}(y))\in B_{1/h}(\bar y)\times (z-1/h,z+1/h)\} \not= \emptyset \qquad \forall k\geq k(h).
\end{equation}
Since we can choose the application $h\mapsto k(h)$ strictly increasing, by \eqref{ab} we can select a sequence
$y_{n_k} \in \supp \mu_{n_k} \cap \{y\in Y: (y,F_{n_k}(y))\in B_{1/h}(\bar y)\times (z-1/h,z+1/h)\}$.
By definition $y_{n_k} \to \bar y$ and $F_{n_k}(y_{n_k})\to z$ as $k\to+\infty$.
Since $F_{n_k}(y_{n_k})$ converges in $\R$ we obtain the bound in \eqref{pointwisebound}.
\end{proof}

\subsection{The extended Wasserstein-Orlicz space $(\Prob{X},W_\psi)$}\label{subsec:Wasserstein}

Given $\mu,\nu \in \Prob{X}$ we define the set of admissible plans $\admissibleplan{\mu}{\nu}$ as follows:
$$ \admissibleplan{\mu}{\nu}:=\{\gamma \in \Prob{X \times X}:\pi^1_\# \gamma =\mu , \, \pi^2_\# \gamma = \nu \},$$
where $\pi^i:X\times X \to X$, for $i=1,2$, are the projections on the first and the second component, defined by $\pi^1(x,y)=x$ and  $\pi^2(x,y)=y$.

Given $\psi$ satisfying \eqref{hp:psi}, the $\psi$-Wasserstein-Orlicz extended distance between $\mu,\nu \in \Prob{X}$ is defined by
\begin{equation}\label{def:WassersteinOrlicz}
\begin{aligned}
	W_\psi(\mu ,\nu ) &:=
\inf_{\gamma \in \admissibleplan{\mu}{\nu}}\inf\left\{\lambda>0:\int _{X \times X} \psi\Big(\frac{\sfd(x,y)}{\lambda}\Big)\,d\gamma (x,y)\leq 1 \right \} \\
& = \inf_{\gamma \in \admissibleplan{\mu}{\nu}}\|\sfd(\cdot,\cdot)\|_{L^\psi_\gamma(X\times X)}.
\end{aligned}
\end{equation}
It is easy to check that
$$ W_\psi(\mu ,\nu )= \inf\left\{\lambda>0:\inf_{\gamma \in \admissibleplan{\mu}{\nu}}\int _{X \times X} \psi\Big(\frac{\sfd(x,y)}{\lambda}\Big)\,d\gamma (x,y)\leq 1 \right \}$$
which is the definition given in \cite{Sturm4} (see also \cite{Kuwada13}).

When the set of $\gamma \in \admissibleplan{\mu}{\nu}$  such that
$\|\sfd(\cdot,\cdot)\|_{L^\psi_\gamma(X\times X)} <+\infty$ is empty,
then $W_\psi(\mu ,\nu )=+\infty$.
Otherwise it is not difficult to show that a minimizer $\gamma \in \admissibleplan{\mu}{\nu}$
in \eqref{def:WassersteinOrlicz} exists.
We denote by $\optimalplan{\mu}{\nu}$ the set of minimizers in \eqref{def:WassersteinOrlicz}.
We observe that
\begin{equation}\label{eq:wopt}
\gamma\in \optimalplan{\mu}{\nu} \qquad \Longleftrightarrow \qquad \int_{X\times X} \psi\Big(\frac{\sfd(x,y)}{W_\psi(\mu,\nu)}\Big)\,d\gamma (x,y)\leq 1.
\end{equation}
Since $\psi$ satisfies \eqref{hp:psi} it is well defined $\psi^{-1}(s)$ for every $s>0$,
with the convention that in the case that $\psi(r)=+\infty$ for $r > r_0$ and $\psi(r_0) < +\infty$
we define $\psi^{-1}(s)= r_0$ for every $s>\psi(r_0)$.

Moreover if  $\gamma\in \optimalplan{\mu}{\nu}$ then
\begin{equation}\label{eq:est}
\int_{X\times X}\sfd(x,y)\,d\gamma (x,y)\leq \psi^{-1}(1) W_\psi(\mu,\nu).
\end{equation}
Indeed, for $\mu\neq\nu$ (the other case is trivial) using Jensen's inequality and \eqref{eq:wopt} %and the monotonicity of $\psi^{-1}$
\begin{align*}
    \psi\Big(\int_{X\times X}\frac{\sfd(x,y)}{W_\psi(\mu,\nu)}\,d\gamma(x,y)\Big)
    \leq \int_{X\times X}\psi\Big(\frac{\sfd(x,y)}{W_\psi(\mu,\nu)}\Big)\,d\gamma(x,y) \leq 1
\end{align*}
and \eqref{eq:est} follows.

Being $(X,\sfd)$ complete, $(\Prob{X},W_\psi)$, is complete too
(the proof of \cite[Proposition 7.1.5]{ags} works also in the case of the extended distance $\sfd$ and the Orlicz-Wasserstein distance).

We observe that $(X,\sfd)$ is embedded in $(\Prob{X},W_\psi)$ via the map $x\mapsto \delta_x$ and it holds
\begin{equation}
    W_\psi(\delta_x,\delta_y) = \frac{1}{\psi^{-1}(1)}d(x,y).
\end{equation}

Thanks to the compatibility condition (iii) in the definition of extended Polish space we also have the following fundamental property:
\begin{equation}\label{wn}
  W_\psi(\mu_n,\mu)\to 0 \quad \Longrightarrow \quad \mu_n \to \mu \text{ narrowly in } \Prob{X}.
\end{equation}

The space $(\Prob{X},W_\psi)$ is an extended Polish space, when in $\Prob{X}$ we consider the narrow topology.

\section{Main theorem}\label{main}
In this section we state and prove our main result: a characterization of absolutely continuous curves
with finite $L^\psi$-energy in the extended $\psi$-Wasserstein-Orlicz space $(\Prob{X},W_\psi)$.
% in terms of probability measures on the space of continuous curves in $X$
% as illustrated in the introduction.

Before to state the result, we define, for every $t\in I$,
the \emph{evaluation map} $e_t :C(I;X) \rightarrow X$ in this way
\begin{equation}\label{evaluation}
	e_t(u )=u (t)
\end{equation}
and we observe that $e_t$ is continuous.

\begin{theorem} \label{th:curveAC}
Let $\psi$ be satisfying \eqref{hp:psi}, \eqref{superlinearity} and \eqref{beatzero}.
Let $(X,\tau,\sfd)$ be an extended Polish space and $I:=[0,T]$, $T>0$.
If $\mu \in AC^{\psi}(I;(\Prob{X},W_\psi))$,
then there exists $\eta \in \Prob{C(I;X)}$ such that
\begin{itemize}
	\item[(i)]$\eta$ is concentrated on $AC^{\psi}(I;(X,\sfd))$,
	\item[(ii)]$(e_t)_\# \eta=\mu_t \qquad \forall t \in I$,
	\item[(iii)]for a.e. $t \in I,$ the metric derivative $|u'|(t)$ exists for $\eta-$a.e. $u\in C(I;X)$ and
                it holds the equality
$$|\mu'|(t) = \||u'|(t)\|_{L^\psi_{\eta}(C(I;X))} \qquad \mbox{for a.e. }t \in I. $$
\end{itemize}
\end{theorem}

\begin{proof}
We preliminary assume that
\begin{equation}\label{hp:velcost}
    |\mu'| = 1 \qquad \text{for a.e. }t\in I,
\end{equation}
and we will remove this assumption in Step 6 of this proof.
We also assume for simplicity that $I=[0,1]$.

For any integer $N \geq 1,$ we divide the unitary interval $I$ in $2^N$ equal parts,
and we denote by $t^i$ the points
\begin{equation*}
	t^i:=\frac{i}{2^N} \qquad i=0,1,\ldots ,2^N.
\end{equation*}
We also denote by $\XX_N$ the product space
$$\XX_N:=X_0 \times X_1 \times \ldots \times X_{2^N}$$
where $X_i,$ with $i=0,1,\ldots ,2^N$, are
$2^N+1$ copies of the same space $X$.

Choosing optimal plans
$$\gamma_N^i \in \optimalplan {\mu_{t^i}}{\mu_{t^{i+1}}} \qquad i=0,1,\ldots ,2^N-1,$$
there exists (see for instance \cite[Lemma 5.3.2 and Remark 5.3.3]{ags})
a measure $\gamma _N \in \Prob{\vett{X}_N}$ such that
\begin{equation*}
	\pi ^i_{\#} \gamma _N = \mu_{t^i} \qquad \mbox{and} \qquad \pi ^{i,i+1}_{\#} \gamma _N = \gamma_ N^i,
\end{equation*}
where we denoted by $\pi^i:\vett{X}_N \to X_i$ the projection on the $i$-th component and
by $\pi^{i,j}:\vett{X}_N \to X_i \times X_j$ the projection on the $(i,j)$-th component.\\
We define ${\sigma}:\XX_N \to \MM(I;X)$, and we use the notation
${\xx}=(x_0,\ldots ,x_{2^N}) \mapsto {\sigma_\sxx}$, by
\begin{equation*}
	{\sigma_\sxx}(t):=x_i \qquad \mbox{if } \quad t \in [t^i,t^{i+1}),  \qquad i=0,1,\ldots ,2^N-1.
\end{equation*}
Finally, we define the sequence of probability measures
$$\eta_N:={\sigma}_{\#}\gamma _N \in \Prob{\cM(I;X)}.$$

{\bf Step 1. (Tightness of $\{\eta _N\}_{N \in \N}$ in $\Prob{\cM(I;X)}$)}
In order to prove the tightness of $\{\eta _N\}_{N \in \N}$  in $\Prob{\cM(I;X)}$
(we recall that $\cM(I;X)$ is a Polish space with the metric $\delta_1$)
it is sufficient to show the existence of a function
$\Phi:\cM(I;X) \rightarrow [0,+\infty]$
whose sublevels
$\lambda_c(\Phi):=\{u \in \cM(I;X): \Phi(u)\leq c\}$ are compact in $\cM(I;X)$ for any $c\in \R_+,$  and
\begin{equation} \label{boundedness}
	\sup_{N \in \N} \int_{\cM(I;X)}\Phi (u) \, d\eta_N (u) <+\infty .
\end{equation}
First of all we observe that $\A :=\{\mu_t : t \in I \}$ is compact
in $(\Prob{X},W_\psi)$ (because it is a continuous image of a compact) and consequently in $\Prob{X}$.
Since, by Prokhorov's Theorem, $\A$ is tight in $\Prob{X}$ there exists a function
$\Psi:X \to [0,+\infty]$ whose sublevels
$\lambda_c(\Psi):=\{x \in X: \Psi(x)\leq c\}$ are compact in $X$ for any $c\in \R_+,$
such that
\begin{equation} \label{intpsi}
	\sup_{t \in I} \int_X \Psi(x)\,d\mu_t(x)<+\infty .
\end{equation}
We define $\Phi : \cM(I;X) \to [0,+\infty]$ by
\begin{equation*}
	\Phi (u ) :=  \int_0^1 \Psi (u(t)) \,dt
						+ \sup_{h\in(0,1)}\int_0^{1-h}\frac{\sfd(u(t+h),u(t))}{h}\,dt .	
\end{equation*}
The compactness of the sublevels $\lambda_c(\Phi)$ in $\cM(I;X)$ follows by Theorem \ref{th:complp}
with the choice $g(x,y)=\sfd(x,y)$.
In order to prove \eqref{boundedness} we begin to show that
\begin{equation}\label{firstbound}
\sup_{N \in \N} \int_{\cM(I;X)} \int_0^1 \Psi (u(t))) \,dt \, d\eta_N (u) <+\infty.
\end{equation}
By the definition of $\eta_N$ we have
\begin{eqnarray*}
	\begin{aligned}
	 \int_{\cM(I;X)} \int_0^1\Psi (u(t)) \,dt\, d\eta_N (u)
	&  = \int_{\XX_N}\int_0^1\Psi(\sigma_\sxx(t))\,dt\,d\gamma_N(\xx)\\
    & =  \int_{\XX_N} \sum_{i=0}^{2^N-1} \int _{t^i}^{t^{i+1}} \Psi(x_i)\,dt\,d\gamma_N(\xx)\\
    & =  \int_{\XX_N} \frac{1}{2^N}\sum_{i=0}^{2^N-1} \Psi(x_i) \,d\gamma_N(\xx)\\
	& = \frac{1}{2^N}\sum_{i=0}^{2^N-1} \int _X \Psi (x)\, d\mu_{t^i}(x) \\
	& \leq  \frac{1}{2^N}\sum_{i=0}^{2^N-1}\sup_{t \in I}\int_X\Psi(x)\,d\mu_t(x) = \sup_{t \in I}\int_X\Psi(x)\,d\mu_t(x)
	\end{aligned}
\end{eqnarray*}
and \eqref{firstbound} follows by \eqref{intpsi}.
The second bound that we have to show is
\begin{equation}\label{terzobound}
	\sup_{N\in\N}\int_{\cM(I;X)}\sup_{h\in(0,1)}\int_0^{1-h}\frac{\sfd(u(t+h),u(t))}{h}\,dt\,d\eta_N(u) <+\infty.
\end{equation}
First of all we prove that for $\xx \in \XX_N$ we have
\begin{equation}\label{ineq:29}
	\sup_{h\in(0,1)}\int_0^{1-h}\frac{\sfd(\sigma_\sxx(t+h),\sigma_\sxx(t))}{h}\,dt
		\leq 2\sum_{i=0}^{2^N-1} \sfd(x_i,x_{i+1}).
\end{equation}
We fix $h\in(0,1)$. When $h<2^{-N}$ we have that $\sigma_\sxx(t+h) = \sigma_\sxx(t)$ for every $t\in [t^i,t^{i+1}-h]$ and $i=0,\ldots,2^N-1$.
Then
\begin{equation}\label{in:40}
	\int_0^{1-h}\sfd(\sigma_\sxx(t+h),\sigma_\sxx(t))\,dt
		= \sum_{i=0}^{2^N-1} \int_{t^i}^{t^{i+1}}\sfd(\sigma_\sxx(t+h),\sigma_\sxx(t))\,dt = h\sum_{i=0}^{2^N-2} \sfd(x_i,x_{i+1}).
\end{equation}
Now we assume that $h \geq 2^{-N}$ and we take the integer $k(h)=[h2^N]$, where $[a]:=\max\{n\in\Z:n\leq a\}$ is the integer part of the real number $a$.
Since the triangular inequality yields
\begin{equation*}
	\sfd(\sigma_\sxx(t+h),\sigma_\sxx(t)) \leq \sum_{i=0}^{k(h)} \sfd(\sigma_\sxx (t+t^{i+1}),\sigma_\sxx (t+t^i)),
\end{equation*}
we have that
\begin{eqnarray}
	\int_0^{1-h} \sfd(\sigma_\sxx(t+h),\sigma_\sxx(t))\,dt &\leq &
		\int_0^{1-t^{k(h)}} \sfd(\sigma_\sxx (t+h),\sigma_\sxx (t))\,dt \nonumber \\
	&\leq & \int_0^{1-t^{k(h)}} \sum_{i=0}^{k(h)} \sfd(\sigma_\sxx (t+t^{i+1}),\sigma_\sxx (t+t^i))\,dt \label{in:50} \\
	&= & \sum_{i=0}^{k(h)} \frac{1}{2^N}\sum_{j=0}^{2^N-k(h)-1}\sfd(x_{i+j+1},x_{i+j}). \nonumber
\end{eqnarray}
Observing that in (\ref{in:50}) the term $\sfd(x_{k+1},x_{k})$, for every $k=0,1,\ldots,2^N-1$, is counted at most $k(h)+1$ times, we obtain that
\begin{equation}\label{in:51}
	\int_0^{1-h}\sfd(\sigma_\sxx(t+h),\sigma_\sxx(t))\,dt \leq \frac{k(h)+1}{2^Nh}h\sum_{j=0}^{2^N-1}\sfd(x_{j+1},x_{j})
        \leq 2h\sum_{j=0}^{2^N-1}\sfd(x_{j+1},x_{j}),
\end{equation}
because
$$\frac{k(h)+1}{h2^N}\leq \frac{k(h)+1}{k(h)}\leq 2.$$
The inequality (\ref{ineq:29}) follows from \eqref{in:51} and \eqref{in:40}.
Finally, by (\ref{ineq:29}), \eqref{eq:est} taking into account the optimality of the plans $\pi^{i,i+1}_\#\gamma_N$,
and \eqref{hp:velcost} we have
\begin{equation}\label{in:53}
\begin{aligned}
    \int_{\cM(I;X)} \sup_{h\in(0,1)}\int_0^{1-h}\frac{\sfd(u(t+h),u(t))}{h}\,dt\,d\eta_N (u) &\leq
	2\int_{\XX_N} \sum_{i=0}^{2^N-1} \sfd(x_i,x_{i+1}) \,d\gamma _N (\xx) \\
    &\leq  2\psi^{-1}(1)\sum _{i=0}^{2^N-1} W_\psi(\mu_{t^{i}},\mu_{t^{i+1}}) \\
	&\leq  2\psi^{-1}(1)\sum _{i=0}^{2^N-1} \frac{1}{2^N} = 2\psi^{-1}(1)
\end{aligned}
\end{equation}
and \eqref{terzobound} follows.

Then, by Prokhorov's Theorem, there exist $\eta \in \Prob{\cM(I;X)}$ and a subsequence $N_n$ such that $\eta_{N_n}\rightarrow \eta$
narrowly in $\Prob{\cM(I;X)}$ as $n \rightarrow +\infty.$

{\bf Step 2. ($\eta$ is concentrated on BV right continuous curves)}
%For $N\in\N$ we use the notation $k(t):=[2^Nt]$, and clearly we have $t^{k(t)}=k(t)/2^N = \max\{i/2^N:i\leq t2^N\}$.
%
%\noindent
We apply Lemma \ref{lemma:Kur} in order to show that $\eta$-a.e. $u\in \supp \eta$ has a right continuous BV representative.

Given a curve $u:[a,b]\to X$, we denote by
$\pV(u,[a,b])=\sup\{\sum_{i=1}^{n}\sfd(u(t_i),u(t_{i+1})):a=t_1<t_2<\ldots<t_n<t_{n+1}=b\}$ its pointwise variation
and by
$\eV(u,[a,b])=\inf\{\pV(w,[a,b]):w(t)=u(t) \text{ for a.e. }t\in (a,b)\}$ its essential variation.

We define $F_N:\cM(I;X)\to [0,+\infty)$ by
\begin{equation}\label{defFN}
  F_N(u)= \begin{cases} \eV(u,I) & \qquad \text{if }u\in\supp \eta_N,\\
                        0 & \qquad \text{if }u\not\in\supp \eta_N.
   \end{cases}
\end{equation}
If $u$ is a.e. equal to $\sigma_\sxx$ then $\eV(u,I)=\pV(\sigma_\sxx,I)$.
Taking into account this equality, the proof of bound \eqref{firstbound} shows that
\begin{equation}\label{boundFN}
  \sup_{N\in\N}\int_{\cM(I;X)} F_N(u)\,d\eta_N(u) <+\infty.
\end{equation}
Since $F_N\geq 0$ by definition, we apply Lemma \ref{lemma:Kur} with the choice $Y=\cM(I;X)$ and $\mu_n=\eta_{N_n}$.
We still denote by $\eta_{N_n}$ the subsequence of $\eta_{N_n}$ given by Lemma \ref{lemma:Kur}.
Let $u\in\supp(\eta)$ be such that \eqref{pointwisebound} holds and we denote by $u_{N_n}\in\supp(\eta_{N_n})$ such that
$u_{N_n}\to u$ in $\cM(I;X)$ and $C$ a constant independent of $n$ such that
\begin{equation}\label{boundFN2}
 F_{N_n}(u_{N_n}) \leq C.
\end{equation}
Moreover, up to extract a further subsequence, we can also assume that $u_{N_n}(t)\to u(t)$ with respect to the distance $\delta$ for a.e. $t\in I$.
Since $u_{N_n}\in\supp(\eta_{N_n})$ we can choose the piecewise constant right continuous representative of $u_{N_n}$,
still denoted by $u_{N_n}$.
From \eqref{boundFN2} we obtain that
\begin{equation}\label{bound1}
 \eV(u_{N_n})=\pV(u_{N_n}) \leq C.
\end{equation}
Defining the increasing functions $v_n:I\to\R$ by $v_n(t)=\pV(u_{N_n},[0,t])$, from the Helly theorem,
up to extract a further subsequence still denoted by $v_n$, there exists an increasing function $v:I\to\R$
such that $v_n(t)$ converges to $v(t)$ for every $t\in I$ (we observe that for \eqref{bound1} $v\leq C$).
Since the set of discontinuity points of $v$ is at most countable we can redefine a right continuous function
$\bar v$ by $\bar v(t)=\lim_{s\to t^+}v(t)$.
Since
\begin{equation}\label{mod1}
 \sfd(u_{N_n}(t),u_{N_n}(s))\leq v_n(s)-v_n(t) \qquad \forall\; t,s\in I, \quad t\leq s,
\end{equation}
from the property \eqref{lscd} it follows that
\begin{equation}\label{mod2}
 \sfd(u(t),u(s))\leq \bar v(s)-\bar v(t) \qquad \text{for a.e. } t,s\in I, \quad t\leq s.
\end{equation}
Since $(X,\sfd)$ is complete, by \eqref{mod2} we can choose the representative of $u$,
$\bar u:I\to X$ defined by $\bar u(t)=\lim_{s\to t^+}u(t)$, which is right continuous by \eqref{mod2}.

We have just proved that $\eta$-a.e. $u\in\supp \eta$
is equivalent (with respect to the a.e. equality) to a $\sfd$-right continuous function with pointwise $\sfd$-bounded variation,
continuous at every points except at most a countable set.

%satisfies the properties (i), (ii), (iii).
%
{\bf Step 3. (Proof of (i))}

Since we want to apply Lemma \ref{lemma:dq},
%In order to show that $\tilde\eta$ is concentrated on $AC^\psi(I;(X,\sfd))$,
we prove that
\begin{equation}\label{sob}
   \sup_{h\in (0,1)} \left\|\frac{\sfd(u(\cdot+h),u(\cdot))}{h}\right\|_{L^\psi(0,1-h)} <+\infty, \quad \mbox{ for } \eta-\mbox{a.e. } u \in \cM(I;X).
\end{equation}

Let us define the sequence of lower semi continuous functions $f_N:\cM(I;X) \to [0,+\infty]$ by
\begin{equation*}
	f_N(u):=\sup_{1/2^N \leq h <1} \int_0^{1-h}\psi\Big(\frac{\sfd(u(t+h),u(t))}{2h}\Big)\,dt,
\end{equation*}
that satisfies the monotonicity property
\begin{equation}\label{mon}
	f_N(u)\leq f_{N+1}(u) \qquad \forall \, u \in \cM(I;X).
\end{equation}
For $h\in[2^{-N},1)$, and $u\in\supp(\eta_N)$, by the monotonicity of $\psi$, the discrete Jensen's inequality and
taking into account that $(k(h)+1)/(2h)\leq 2^N$, we have that
\begin{align*}
    & \int_{0}^{1-h}\psi\Big(\frac{\sfd(u(t+h),u(t))}{2h}\Big)\,dt  \\
    & = \int_{0}^{1-h}\psi\Big(\frac{\sfd(x_{k(t+h)},x_{k(t)})}{2h}\Big)\,dt \\
    & = \int_{0}^{1-t^{k(h)}}\psi\Big(\frac{\sfd(x_{k(t+h)},x_{k(t)})}{2h}\Big)\,dt \\
    &\leq \int_{0}^{1-t^{k(h)}}\psi\Big(\frac{1}{k(h)+1}\sum_{i=0}^{k(h)}\frac{k(h)+1}{2h}\sfd(x_{k(t)+i+1},x_{k(t)+i})\Big)\,dt \\
    &\leq \int_{0}^{1-t^{k(h)}}\frac{1}{k(h)+1}\sum_{i=0}^{k(h)}\psi\Big(\frac{k(h)+1}{2h}\sfd(x_{k(t)+i+1},x_{k(t)+i})\Big)\,dt \\
    & = \sum_{j=0}^{2^N-k(h)-1} 2^{-N}\frac{1}{k(h)+1}\sum_{i=0}^{k(h)}\psi\Big(2^N\sfd(x_{j+i+1},x_{j+i})\Big) \\
    & \leq \sum_{j=0}^{2^N-1} 2^{-N} \psi\Big(2^N\sfd(x_{j+1},x_{j})\Big).
\end{align*}
It follows that
$$ f_N(u) \leq \sum_{j=0}^{2^N-1} 2^{-N} \psi\Big(2^N\sfd(x_{j+1},x_{j})\Big) $$
for every $u \in\supp(\eta_N)$.
Integrating the last inequality, taking into account that $W_\psi(\mu_{t^j},\mu_{t^{j+1}})\leq 2^{-N}$ and
$$\int_{\XX_N}\psi\Big(\frac{\sfd(x_{j+1},x_{j})}{W_\psi(\mu_{t^{j+1}},\mu_{t^{j}})}\Big) \, d\gamma_N(\xx) \leq 1,$$
we obtain that
\begin{align*}
    \int_{\cM(I;X)}f_N(u) \, d\eta_N (u)
   & \leq \sum_{j=0}^{2^N-1} 2^{-N} \int_{\XX_N}\psi\Big(2^N\sfd(x_{j+1},x_{j})\Big)\, d\gamma_N(\xx)\\
   & \leq \sum_{j=0}^{2^N-1} 2^{-N} \int_{\XX_N}\psi\Big(\frac{\sfd(x_{j+1},x_{j})}{W_\psi(\mu_{t^{j+1}},\mu_{t^{j}})}\Big)\, d\gamma_N(\xx) \leq 1.
\end{align*}

The lower semi continuity of $f_N,$
%is a lower semi continuous function from $L^p(I;X)$ to $[0,+\infty]$ because the functions
%$$ u \mapsto \int_0^{1-h} \frac{d^p(u(t+h),u(t))}{h^p}\,dt $$ are continuous for every $h.$\\
the monotonicity \eqref{mon} of $f_N$ and the last inequality yield
\begin{equation*}
	\int_{\cM(I;X)}f_N(u) \, d\eta (u) \leq 1  \qquad \forall \, N \in \N,
\end{equation*}
and consequently, by monotone convergence Theorem, we have that
$$\int_{\cM(I;X)} \sup_{N \in \N} f_N(u) \, d\eta (u) \leq 1,$$
and
\begin{equation}\label{conc}
		\sup_{N \in \N} f_N(u) <+\infty \quad \mbox{ for } \eta-\mbox{a.e. } u \in \cM(I;X).
\end{equation}

Since
$$ \sup_{N \in \N} f_N(u) = \sup_{0 < h <1} \int_0^{1-h} \psi\Big(\frac{\sfd(u(t+h),u(t))}{2h}\Big)\,dt ,$$
and
$\int_0^{1-h} \psi\Big(\frac{\sfd(u(t+h),u(t))}{2h}\Big)\,dt \leq C$ implies
$\left\|\frac{\sfd(u(\cdot+h),u(\cdot))}{h}\right\|_{L^\psi(0,1-h)}\leq \max\{C,1\}$
we obtain \eqref{sob}.

Finally, taking into account Step 2, we can associate to $\eta$-a.e. $u\in\supp\eta$ a right continuous representative $\bar u$, with
at most a countable points of discontinuity
satisfying \eqref{eq:sobolev}.
By Lemma \ref{lemma:dq} this representative belongs to $AC^\psi(I;(X,\sfd)).$

Defining the canonical immersion $T:C(I;X)\to \cM(I;X)$ and observing that it is continuous,
we define the new Borel probability measure $\tilde\eta\in\Prob{C(I;X)}$ by
$\tilde\eta(B)=\eta(T(B))$.
For the previous steps $\tilde \eta$ is concentrated on $AC^\psi(I;(X,\sfd)).$

{\bf Step 4. (Proof of (ii))}
In order to show (ii) we prove that for every $t\in I,$
\begin{equation}\label{etaeqmut}
	\int_{C(I;X)} \varphi(u(t))\,d\tilde\eta(u) =\int_X\varphi(x)\,d\mu_{t}(x)
			\qquad \forall \varphi \in C_b(X).
\end{equation}
Let $\varphi \in C_b(X)$. Since $g:I\to\R$ defined by
$$g(t):= \int_X \varphi (x) \, d\mu_{t}(x) $$
is uniformly continuous in $I$, we have that
the sequence of piecewise constant functions $g_N:I\to\R$ defined by
$$g_N(t):= g(t^i) = \int_X \varphi (x) \, d\mu_{t^i}(x) \qquad \mbox{if } t \in [t^i,t^{i+1}) ,$$
converges uniformly to $g$ in $I$ when $N \to +\infty$.
Then, for every test function $\zeta \in C_b(I),$ we have that
\begin{equation}
	\lim_{N \to +\infty} \int_0^1 \zeta(t) g_N(t)\,dt=\int_0^1 \zeta(t) g(t)\,dt.
\end{equation}
On the other hand
\begin{eqnarray*}
	\int_0^1 \zeta(t) g_N(t)\,dt &=& \int_0^1 \zeta(t) \int _{\cM(I;X)} \varphi (u(t)) \, d\eta_{N}(u) \,dt \\
	&=& \int _{\cM(I;X)} \int_0^1 \zeta(t) \varphi (u(t)) \,dt\, d\eta_{N}(u) .
\end{eqnarray*}
Since the map
$$u \mapsto \int_0^1 \zeta (t) \varphi (u(t)) \, dt$$ is continuous and bounded from $\cM(I;X)$ to $\R,$
then by the narrow convergence of $\eta_{N_n}$ we have
\begin{equation*}
	\lim_{n \to +\infty} \int _{\cM(I;X)} \int_0^1 \zeta (t) \varphi (u(t))\,dt\,d\eta_{N_n}(u)=
	\int _{\cM(I;X)} \int_0^1 \zeta(t)\varphi(u(t))\,dt\,d\eta(u).
\end{equation*}
By Fubini's Theorem and the definition of $\tilde\eta$
\begin{eqnarray*}
	\int_{\cM(I;X)}\int_0^1 \zeta(t) \varphi(u(t))\,dt\,d\eta(u) & = &
	\int_{C(I;X)}\int_0^1\zeta(t)\varphi(u(t))\,dt\,d\tilde\eta(u )\\
	 & = & \int_0^1\zeta(t)\int_{C(I;X)}\varphi(u(t))\,d\tilde\eta(u)\,dt.
\end{eqnarray*}
By the uniqueness of the limit then
$$ \int_0^1 \zeta (t) \int _{C(I;X)} \varphi (u (t)) \, d\tilde \eta (u ) \, dt =
	\int_0^1 \zeta (t) \int_X \varphi (x) \, d\mu_{t}(x) \, dt  \qquad \forall \zeta \in C_b(I) ,$$
from which
\begin{equation}\label{a.e.etetaeqmut}
	\int_{C(I;X)}\varphi(u(t))\,d\tilde\eta(u) = \int_X\varphi(x)\,d\mu_{t}(x)
			\qquad \mbox{for a.e. } t \in I.
\end{equation}
Since the applications $t \mapsto \int_X \varphi (x) \, d\mu_{t}(x)$ and
$t \mapsto \int_{C(I;X)} \varphi (u (t)) \, d\tilde \eta (u )$ are continuous,
(\ref{a.e.etetaeqmut}) is true for every $t \in I$ and (\ref{etaeqmut}) is proved.

{\bf Step 5. (Proof of (iii))}

First of all we check that for a.e. $t \in I,$ $|u '|(t)$ exists for $\tilde\eta-$a.e. $u \in C(I;X).$
We set $\Lambda :=\{(t,u) \in I \times C(I;X) : |u '|(t) \mbox{ does not exist}\}$.
$\Lambda$ is a Borel subset of $I \times C(I;X)$ since
the maps $G_h:I \times C(I;X)\to\R$ defined by $G_h(t,u):= \frac{d(u (t+h),u (t))}{|h|}$ are lower semi continuous
for every $h\not =0$, and $\Lambda=\{(t,u) \in I \times C(I;X) : \liminf_{h\to 0}G_h(t,u)<\limsup_{h\to 0}G_h(t,u)\}$.
Since $\tilde\eta$ is concentrated on $AC(I;(X,\sfd))$ curves, we have that for $\tilde\eta$-a.e. $u\in C(I;X),$
$\Leb^1(\{t \in I:(t,u) \in \Lambda \}) =0$  and
then Fubini's Theorem implies that for a.e. $t \in I,$  $\tilde\eta(\{u \in C(I;X):(t,u) \in \Lambda \})=0.$

Let $a,b \in I$ such that $a<b$ and let $h>0$ such that $b+h\in I$.
Recalling that $k(h)=[2^Nh]$, for every $N\in\N$ such that $2^{-N}\leq h$, by the monotonicity of $\psi$ and the discrete Jensen's inequality we have
\begin{align*}
    &\int_{\cM(I;X)} \int_{a}^{b}\psi\Big(\frac{k(h)}{k(h)+1}\frac{\sfd(u(t+h),u(t))}{h}\Big)\,dt \, d\eta_N(u) \\
    &\leq \int_{\XX_N} \int_{a}^{b}\psi\Big(\frac{k(h)}{k(h)+1}\frac{\sfd(x_{k(t+h)},x_{k(t)})}{h}\Big)\,dt \, d\gamma_N(\xx) \\
    &\leq\int_{\XX_N} \int_{a}^{b}\psi\Big(\frac{1}{k(h)+1}\sum_{i=0}^{k(h)}\frac{k(h)}{h}\sfd(x_{k(t)+i+1},x_{k(t)+i})\Big)\,dt \, d\gamma_N(\xx) \\
    &\leq\int_{\XX_N} \int_{a}^{b}\sum_{i=0}^{k(h)}\frac{1}{k(h)+1}\psi\Big(\frac{k(h)}{h}\sfd(x_{k(t)+i+1},x_{k(t)+i})\Big)\,dt \, d\gamma_N(\xx).
\end{align*}
Since $k(h)/h\leq 2^N$ and, by \eqref{hp:velcost}, $W_\psi(\mu_{t^k},\mu_{t^{k+1}})\leq 2^{-N}$, we have that
\begin{align*}
    & \int_{\XX_N} \int_{a}^{b}\sum_{i=0}^{k(h)}\frac{1}{k(h)+1}\psi\Big(\frac{k(h)}{h}\sfd(x_{k(t)+i+1},x_{k(t)+i})\Big)\,dt \, d\gamma_N(\xx) \\
    & \leq \int_{\XX_N} \int_{a}^{b}\sum_{i=0}^{k(h)}\frac{1}{k(h)+1}\psi\Big(2^N\sfd(x_{k(t)+i+1},x_{k(t)+i})\Big)\,dt \, d\gamma_N(\xx) \\
    & \leq \int_{a}^{b}\sum_{i=0}^{k(h)}\frac{1}{k(h)+1}\int_{\XX_N}
            \psi\Big(\frac{\sfd(x_{k(t)+i+1},x_{k(t)+i})}{W_\psi(\mu_{t^{k(t)+i+1}},\mu_{t^{k(t)+i}})}\Big) \, d\gamma_N(\xx) \,dt \leq b-a ,
\end{align*}
where we used the inequality
$$\int_{\XX_N}\psi\Big(\frac{\sfd(x_{k(t)+i+1},x_{k(t)+i})}{W_\psi(\mu_{t^{k(t)+i+1}},\mu_{t^{k(t)+i}})}\Big) \, d\gamma_N(\xx) \leq 1.$$
It follows that
\begin{equation*}
    \int_{\cM(I;X)}\frac{1}{b-a} \int_{a}^{b}\psi\Big(\frac{k(h)}{k(h)+1}\frac{\sfd(u(t+h),u(t))}{h}\Big)\,dt \, d\eta_N(u) \leq 1
\end{equation*}
and then, passing to the limit along the sequence $\eta_{N_n}$,
\begin{equation*}
    \int_{C(I;X)}\frac{1}{b-a} \int_{a}^{b}\psi\Big(\frac{\sfd(u(t+h),u(t))}{h}\Big)\,dt \, d\tilde\eta(u) \leq 1.
\end{equation*}
Taking into account (i), Fubini's Theorem and Lebesgue differentiation Theorem we obtain
\begin{equation}\label{ine}
    \int_{C(I;X)}\psi\Big(|u'|(t)\Big) \, d\tilde\eta(u) \leq 1 \qquad \text{for a.e. }t\in I
\end{equation}
and this shows that
$$\||u'|(t)\|_{L^\psi_{\tilde\eta}(C(I;X))} \leq 1 = |\mu'|(t) \qquad \text{for a.e. }t\in I.$$

{\bf Step 6. (Conclusion)}
Finally we have to remove the assumption \eqref{hp:velcost}.
Let $\mu \in AC^{\psi}(I;(\Prob{X},W_\psi))$ with length $L:=\int_0^T |\mu'|(t)\,dt$.

If $L=0$, then $\mu_t=\mu_0$ for every $t\in I$ and $\mu$ is represented by $\eta:= \sigma_{\#}\mu_0$,
where $\sigma:X\to C(I;X)$ denotes the function $\sigma(x)=c_x$, $c_x(t):=x$ for every $t\in I$.

When $L>0$ we can reparametrize $\mu$ by its arc-length (see Lemma 1.1.4(b) of \cite{ags} for the details).
We define the increasing function $\ss:I\to[0,L]$ by $\ss(t):=\int_0^t|\mu'|(r)\,dr$ observing that
$\ss$ is absolutely continuous with pointwise derivative
\begin{equation}\label{aeder}
\ss'(t)=|\mu'|(t) \qquad \text{for a.e. }t\in I.
\end{equation}
Defining $\ss^{-1}:I\to[0,L]$ by $\ss^{-1}(s)=\min\{t\in I: \ss(t)=s\}$ it is easy to check that
the new curve $\hat\mu:[0,L]\to\Prob{X}$ defined by $\hat\mu_s=\mu_{\sss^{-1}(s)}$ satisfies
$|\hat\mu'|(s)=1$ for a.e. $s\in [0,L]$ and $\mu_t=\hat\mu_{\sss(t)}$.
By the previous steps, we represent $\hat\mu$ by a measure $\hat\eta$ concentrated on $AC^{\psi}([0,L];(X,\sfd))$.
Denoting by $F:C([0,L];X)\to C(I;X)$ the map defined by $F(\hat u)=\hat u\circ\ss$, we represent $\mu$ by $\eta:=F_\#\hat\eta$.
Clearly $(e_t)_\#\eta=(e_t\circ F)_\#\hat\eta = \hat\mu_{\sss(t)} = \mu_t$.
Moreover, $\eta$ is concentrated on curves $u$ of the form $u(t)=\hat u(\ss(t))$ with $\hat u \in AC^{\psi}([0,L];(X,\sfd))$.
Since $\ss$ is monotone and $AC(I;\R)$ and $\hat u$ is $AC([0,L];(X,\sfd))$ then $\hat u\circ\ss$ is $AC(I;(X,\sfd))$,
and the metric derivative satisfies
\begin{equation}\label{mdc}
|u'|(t)\leq |\hat u'|(\ss(t))\ss'(t) \qquad \text{for a.e. }t\in I.
\end{equation}

Let $t\in I$ such that $\ss'(t)$ and $|\mu'|(t)$ exist and $\ss'(t)=|\mu'|(t)>0$.
Taking into account \eqref{ine} and Jensen's inequality we have for $h>0$
\begin{align*}%\label{ine2}
    \int_{C(I;X)}\psi\Big(\frac{\sfd(u(t+h),u(t))}{\ss(t+h)-\ss(t)}\Big) \, d\eta(u)
    & = \int_{C([0,L];X)}\psi\Big(\frac{\sfd(\hat u(\ss(t+h)),u(\ss(t)))}{\ss(t+h)-\ss(t)}\Big) \, d\hat\eta(\hat u) \\
    & \leq \int_{C([0,L];X)}\psi\Big(\frac{1}{\ss(t+h)-\ss(t)}\int_{\ss(t)}^{\ss(t+h)} |\hat u'|(r)\,dr \Big) \, d\hat\eta(\hat u) \\
    & \leq \frac{1}{\ss(t+h)-\ss(t)}\int_{\ss(t)}^{\ss(t+h)} \int_{C([0,L];X)}\psi\Big( |\hat u'|(r) \Big)\,d\hat\eta(\hat u) \,dr
    \leq 1.
\end{align*}
By Fatou's lemma, taking into account that $\eta$ is concentrated on $AC(I;(X,\sfd))$ curves, we obtain the inequality
\begin{equation}\label{md}
    \int_{C(I;X)}\psi\Big(\frac{|u'|(t)}{|\mu'|(t)}\Big) \, d\eta(u) \leq 1 .
\end{equation}

On the other hand, if $|\mu'|(t)=0$ on a set $J\subset I$ of positive measure, then for $\eta$-a.e. $u$ we have
$|u'|(t)=0$ for a.e. $t\in J$ because of the inequality \eqref{mdc}. 
Taking into account this observation and \eqref{md} we obtain the inequality
\begin{equation}\label{ineq}
 \||u'|(t)\|_{L^\psi_{\eta}(C(I;X))} \leq |\mu'|(t), \qquad \text{for a.e. }t\in I.
\end{equation}

We prove that $\eta$ is concentrated on $AC^{\psi}(I;(X,\sfd))$.
For every $v\in L^{\psi^*}(I)$, $v\geq 0$, $\|v\|_{L^{\psi^*}(I)}\leq 1$, from \eqref{ineq}
we have that
$$ \int_I \||u'|(t)\|_{L^\psi_{\eta}(C(I;X))}v(t) \,dt \leq \int_I|\mu'|(t)v(t)\,dt. $$
By the inequality \eqref{duality} it follows that,
for every $w\in L^{\psi^*}_{\eta}(C(I;X))$, $w\geq 0$, $\|w\|_{L^{\psi^*}_{\eta}(C(I;X))}\leq 1$,
$$ \int_I \int_{C(I;X)} |u'|(t)w(u)\,d\eta(u) v(t) \,dt \leq 4 \||\mu'|\|_{L^{\psi}(I)}. $$
By Fubini Theorem and \eqref{duality} we obtain that
$$  \| \||u'|(t)\|_{L^{\psi}(I)}\|_{L^\psi_{\eta}(C(I;X))} \leq 4 \||\mu'|\|_{L^{\psi}(I)}$$
and (i) holds.

In order to show the opposite inequality of \eqref{ineq}, we assume that
$t \in I$ is such that $|u'|(t)$ exists for $\eta-$a.e. $u\in C(I;X)$
and $\lambda_t:= \||u'|(t)\|_{L^\psi_\eta(C(I;X))}>0$.
We fix $\eps>0$. Since $\int_{C(I;X)} \psi\Big(\frac{|u'|(t)}{\lambda_t}\Big)\,d\eta(u)\leq 1$
and $\psi$ is strictly increasing on an interval of the form $(r_0,r_1)$ where $r_0\geq 0$, $r_1\leq+\infty$
and $\psi(r)=0$ for $r<r_0$, $\psi(r)=+\infty$ for $r>r_1$,
we have that $$\int_{C(I;X)} \psi\Big(\frac{|u'|(t)}{\lambda_t+\eps}\Big)\,d\eta(u)<1.$$
For $h>0$, let $\gamma_{t,t+h}:=(e_t,e_{t+h})_\#\eta$.
Taking into account that $\eta$ is
concentrated on $AC^{\psi}(I;(X,\sfd))$, we have
\begin{equation}\label{a}
\begin{aligned}
	\limsup_{h\to 0^+}\int_{X\times X} \psi\Big(\frac{\sfd(x,y)}{h(\lambda_t+\eps)}\Big)\, d\gamma_{t,t+h}(x,y)
    & = \limsup_{h\to 0^+} \int_{C(I;X)} \psi\Big(\frac{\sfd(u(t),u(t+h))}{h(\lambda_t+\eps)}\Big)\,d\eta(u) \\
        & \leq	\int_{C(I;X)} \limsup_{h\to 0^+} \psi\Big(\frac{\sfd(u(t),u(t+h))}{h(\lambda_t+\eps)}\Big)\,d\eta(u)  \\
		& = \int_{C(I;X)}\psi\Big(\frac{|u '|(r)}{\lambda_t+\eps}\Big)\,d\eta(u) < 1.
\end{aligned}
\end{equation}
Consequently there exists $\bar h$ (depending on $\eps$ and $t$) such that
\begin{equation*}\label{}
	\int_{X\times X} \psi\Big(\frac{\sfd(x,y)}{h(\lambda_t+\eps)}\Big)\, d\gamma_{t,t+h}(x,y) \leq 1 \qquad \forall h \in (0,\bar h).
\end{equation*}
Since $\gamma_{t,t+h} \in \admissibleplan{\mu_t}{\mu_{t+h}}$,
the last inequality  shows that
\begin{equation*}\label{aa}
	W_\psi(\mu_t,\mu_{t+h})\leq h(\lambda_t+\eps) \qquad \forall h \in (0,\bar h).
\end{equation*}
Finally, dividing by $h$ and passing to the limit for $h\to 0^+$ we obtain
\begin{equation*}\label{estimate}
		|\mu'|(t) \leq \||u'|(t)\|_{L^\psi_\eta(C(I;X))} \qquad \mbox{for a.e. } t \in I.
	\end{equation*}
\end{proof}

\begin{remark}
The following example shows that the assumptions on $\psi$ are necessary for the validity of Theorem \ref{th:curveAC}.

Since $\psi$ is convex, if \eqref{superlinearity} and \eqref{beatzero} are not satisfied there exist $a,b\in\R$ such that $0<a\leq b<+\infty$
and $at\leq \psi(t)\leq bt$ for every $t\geq 0$.
Then it holds $a W_1(\mu,\nu)\leq W_\psi(\mu,\nu)\leq b W_1(\mu,\nu)$, where $W_1$ denotes the distance $W_{\phi}$ for $\phi(t)=t$.
% We consider the space $X=\R^n$ endowed with the euclidean distance.
Given two distinct points $x_0,x_1\in X$, consider the curve $\mu:[0,1]\to \Prob{X}$ defined by $\mu_t=(1-t)\delta_{x_0}+t\delta_{x_1}$.
We observe that $\supp(\mu_t)=\{x_0,x_1\}$ for $t\in(0,1)$ and $\supp(\mu_i)=\{x_i\}$ for $i=0,1$.
Clearly $\mu$ is Lipschitz with respect to the distance $W_\psi$ and in particular $\mu\in AC^{\psi}(I;X)$.
If there is a measure $\eta$ satisfying properties (i) and (ii) of Theorem \ref{th:curveAC},
then for $\eta$-a.e. $u$ there holds $u(i)=x_i$ for $i=0,1$ and $u(t)\in \{x_0,x_1\}$ for every $t\in(0,1)$
and $u$ cannot be continuous.
\end{remark}

\section{Geodesics in $(\Prob{(X,\sfd)},W_\psi)$}\label{sec:geodesics}

%As in the paper \cite{lisini07} w
We apply Theorem \ref{th:curveAC} in order to characterize the geodesics of
the metric space $(\Prob{X},W_\psi)$ in terms of the geodesics of the space $(X,\sfd)$.

In this section $I$ denotes the unitary interval $[0,1].$

%Recalling that the length of $u\in AC(I;X)$ is defined by $L(u):=\int_0^1 |u'|(t)\,dt,$
%we say that $X$ is a \emph{length space} if for every $x,y \in X,$
%\begin{equation}\label{def:ls}
%d(x,y)=\inf \{L(u):u\in AC(I;X),\,u(0)=x,\,u(1)=y\}
%\end{equation}
%and $X$ is a \emph{geodesic space} if for every $x,y \in X,$ the $\inf$ in (\ref{def:ls}) is a minimum,
%i.e. there exists
%$u\in AC(I;X)$ such that
%\begin{equation}\label{ls}
%	 u(0)=x, \quad u(1)=y, \quad d(x,y)=L(u).
%\end{equation}
%We call $u$ satisfying (\ref{ls}) \emph{minimizing geodesic} of $X.$
We say that $u:I\rightarrow X$ is a constant speed geodesic in $(X,\sfd)$ if
\begin{equation}\label{constspeed}
	\sfd(u(t),u(s))=|t-s|\sfd(u(0),u(1)) \qquad \forall s,t \in I.
\end{equation}

%If $u$ is a constant speed minimizing geodesic then it is a minimizing geodesic,
%since $|u '|(t)=d(u(0),u(1))$ for every $t\in I.$
%Conversely every minimizing geodesic can be reparametrized in such a way (\ref{constspeed}) holds.\\

We define the set $$G(X,\sfd):=\{u:I\to X: u \text{ is a  constant speed geodesic of }(X,\sfd)\}.$$
%and we observe that it is immediate to check that $G$ is a closed subset of $\Gamma.$\\

\begin{proposition}
Let $(X,\tau,\sfd)$ be an extended Polish space and $\psi$ be satisfying \eqref{hp:psi}.
If $\eta \in \Prob{C(I;X)}$ is concentrated on $G(X,\sfd)$ and $\gamma_{0,1}:=(e_0,e_1)_\#\eta\in\optimalplan{(e_0)_\#\eta}{(e_1)_\#\eta}$,
then the curve $\mu:I\to \Prob{X}$ defined by $\mu_t=(e_t)_\#\eta$ is a constant speed geodesic in $(\Prob{X},W_\psi)$.
\end{proposition}

\begin{proof}
Since $\gamma_{0,1}:=(e_0,e_1)_\#\eta\in\optimalplan{\mu_0}{\mu_1}$,
the following inequality holds
\begin{equation}\label{in:60}
\int_{X\times X} \psi\Big(\frac{\sfd(x,y)}{W_\psi(\mu_0,\mu_1)}\Big)\, d\gamma_{0,1}(x,y) \leq 1.
\end{equation}
Since $\eta$ is concentrated on constant speed geodesics and $\gamma_{s,t}:=(e_s,e_t)_\#\eta\in\admissibleplan{\mu_s}{\mu_t}$ we have,
for every $t,s\in I$, $t\not=s$.
\begin{equation}\label{in:61}
\begin{aligned}
\int_{X\times X} \psi\Big(\frac{\sfd(x,y)}{W_\psi(\mu_0,\mu_1)}\Big)\, d\gamma_{0,1}(x,y)
 &= \int_{C(I;X)} \psi\Big(\frac{\sfd(u(0),u(1))}{W_\psi(\mu_0,\mu_1)}\Big)\,d\eta(u) \\
&= \int_{C(I;X)} \psi\Big(\frac{\sfd(u(t),u(s))}{|t-s|W_\psi(\mu_0,\mu_1)}\Big)\,d\eta(u) \\
&= \int_{X\times X} \psi\Big(\frac{\sfd(x,y)}{|t-s|W_\psi(\mu_0,\mu_1)}\Big)\, d\gamma_{t,s}(x,y).
\end{aligned}
\end{equation}
From \eqref{in:60} and \eqref{in:61} it follows that
\begin{equation}\label{ingeod}
    W_\psi(\mu_t,\mu_s)\leq |t-s|W_\psi(\mu_0,\mu_1) \qquad \forall s,t \in I.
\end{equation}
By the triangular inequality we conclude that equality holds in \eqref{ingeod}.
\end{proof}

\begin{theorem}
Let $(X,\tau,\sfd)$ be an extended Polish space and $\psi$ be satisfying \eqref{hp:psi}, \eqref{superlinearity} and \eqref{beatzero}.
Let $\mu:I\to \Prob{X}$ be a constant speed geodesic in $(\Prob{X},W_\psi)$
and $\eta \in \Prob{C(I;X)}$ a measure representing $\mu$ in the sense that (i), (ii) and (iii) of Theorem \ref{th:curveAC} hold.
Then $\gamma_{s,t}:=(e_s,e_t)_\#\eta$ belongs to $\optimalplan{\mu_s}{\mu_t}$ for every $s,t\in I$.
If, in addition, $\psi$ is strictly convex and
\begin{equation}\label{equalitynorm}
    \int_{X\times X} \psi\Big(\frac{\sfd(x,y)}{W_\psi(\mu_0,\mu_1)}\Big)\, d\gamma_{0,1}(x,y) = 1 ,
\end{equation}
then $\eta$ is concentrated on $G(X,\sfd)$.
\end{theorem}

\begin{proof}
Let $L=W_\psi(\mu_0,\mu_1)$.
Since $\mu$ is a constant speed geodesic and (iii) of Theorem \ref{th:curveAC} holds
\begin{equation}\label{L}
    L = |\mu'|(r) =   \||u'|(r)\|_{L^\psi_{\eta}(C(I;X))} \qquad \mbox{for a.e. }r \in I.
\end{equation}
Let $t,s\in I$, $t\not=s$.
Since, by \eqref{L}, it holds
\begin{equation*}
    \frac{1}{t-s}\int_s^t\int_{C(I;X)} \psi\Big(\frac{|u'|(r)}{L}\Big)\,d\eta(u)\,dr \leq 1 ,
\end{equation*}
Fubini's theorem and Jensen's inequality yield
\begin{equation}\label{in:62}
    \int_{C(I;X)} \psi\Big(\frac{1}{t-s}\int_s^t\frac{|u'|(r)}{L}\,dr\Big)\,d\eta(u) \leq 1.
\end{equation}
By the monotonicity of $\psi$ and \eqref{in:62} we obtain
\begin{equation*}
    \int_{C(I;X)} \psi\Big(\frac{\sfd(u(s),u(t))}{|t-s|L}\Big)\,d\eta(u) \leq 1.
\end{equation*}
Since $|t-s|L= W_\psi(\mu_s,\mu_t)$ we have
\begin{equation}
    \int_{C(I;X)} \psi\Big(\frac{\sfd(u(s),u(t))}{W_\psi(\mu_s,\mu_t)}\Big)\,d\eta(u) \leq 1
\end{equation}
and, recalling \eqref{eq:wopt}, this shows that $\gamma_{s,t}$ is optimal.

Assuming \eqref{equalitynorm} and using Jensen's inequality we have
\begin{equation}\label{equality}
\begin{aligned}
    1 &= \int_{C(I;X)} \psi\Big(\frac{\sfd(u(0),u(1))}{L}\Big)\,d\eta(u)
    \leq \int_{C(I;X)} \psi\Big(\int_0^1\frac{|u'|(t)}{L}\,dt\Big)\,d\eta(u) \\
    &\leq \int_{C(I;X)} \int_0^1\psi\Big(\frac{|u'|(t)}{L}\Big)\,dt\,d\eta(u)
    = \int_0^1\int_{C(I;X)}\psi\Big(\frac{|u'|(t)}{L}\Big)\,d\eta(u)\,dt \leq 1.
\end{aligned}
\end{equation}
It follows that equality holds in \eqref{equality} and, still by Jensen's inequality, we have
\begin{equation}\label{eq:e}
\psi\Big(\int_0^1\frac{|u'|(t)}{L}\,dt\Big)=\int_0^1\psi\Big(\frac{|u'|(t)}{L}\Big)\,dt, \qquad \text{ for } \eta\text{-a.e. }u\in C(I;X) .
\end{equation}
The strict convexity of $\psi$ implies that, if $u$ satisfies the equality in \eqref{eq:e}, then $|u'|$ is constant, 
say $|u'|(t)=L_u$ for a.e. $t\in I$.
Analogously equality in \eqref{equality} shows that $\psi\Big(\frac{\sfd(u(0),u(1))}{L}\Big)= \psi\Big(\frac{L_u}{L}\Big)$ 
for  $\eta$-a.e. $u\in C(I;X)$.
The strict monotonicity of $\psi$ implies that $\sfd(u(0),u(1))= L_u$ and we conclude that $u\in G(X,\sfd)$ for  $\eta$-a.e. $u\in C(I;X)$.
\end{proof}

\section*{Acknowledgements}
The author would like to thank Luigi Ambrosio, Nicola Gigli and Giuseppe Savar\'e for helpful conversations on the topic of this paper.
The author is supported by a MIUR-PRIN 2010-2011 grant for the project “Calculus of Variations”.
The author is member of the GNAMPA group of the Istituto Nazionale di Alta Matematica (INdAM).
%\bibliographystyle{alpha}
%\bibliography{articolo}

%
\end{document}